\documentclass[a4paper,11pt]{amsart}

\pdfoutput=1
\usepackage{amsmath,amssymb,amsthm,amsfonts}

\usepackage{times}
\usepackage{enumerate}
\usepackage{layout}
\usepackage{verbatim}
\usepackage{epsfig}
\usepackage{graphicx}
\usepackage{pgf}
\usepackage{caption}
\usepackage{subcaption}
\usepackage{xfrac}

\newtheorem{theorem}{Theorem}[section]
\newtheorem*{theorem*}{Theorem}
\newtheorem{proposition}[theorem]{Proposition}
\newtheorem{lemma}[theorem]{Lemma}

\newtheorem*{corollary*}{Corollary}

\newtheorem*{conjecture*}{Conjecture}

\theoremstyle{remark}

\newtheorem*{acknowledgments}{Acknowledgments}

\theoremstyle{definition}

\numberwithin{figure}{section}

\newcounter{fig}

\newcommand{\vb}{\vspace{2mm}}
\newcommand{\DD}{{\rm d}}
\newcommand{\hmm}{\hspace{-0mm}}

\def\R{\mathbb R}

\def\E{\mathbb E}

\renewcommand{\vec}[1]{\boldsymbol{#1}} 
\newcommand{\COV}{{\mathbb C}{\rm ov}}
\newcommand{\VAR}{{\mathbb V}{\rm ar}}

\usepackage{color}
\definecolor{light}{gray}{0.8}

\def\ERASE#1{\textcolor{light}{#1}}

\usepackage[T1]{fontenc}

\begin{document}

\title
[A functional CLT for a Markov-modulated infinite-server queue] {A functional central limit theorem for \\a Markov-modulated infinite-server queue}

\author{D.\ Anderson$^{\lowercase{a}}$, j.\ Blom$\,^{\lowercase{c}}$, M.\ Mandjes$\,^{\lowercase{b,c}}$, H.\ Thorsdottir$\,^{\lowercase{c,b,\star}}$, K. de Turck$^{\lowercase{d}}$}


\date{\today}

\begin{abstract}


 The production of molecules in a chemical reaction network is modelled as a Poisson process with a Markov-modulated arrival rate and an exponential decay rate.
We analyze the distributional properties of $M$, the number of molecules, under specific time-scaling; the background process is sped up by $N^{\alpha}$, the arrival rates are scaled by $N$, for $N$ large.
A functional central limit theorem is derived for $M$, which after centering and scaling, converges to an Ornstein-Uhlenbeck process. A dichotomy depending on $\alpha$ is observed. For $\alpha\leq1$ the parameters of the limiting process contain the deviation matrix associated with the background process.

\vb


\vb

\begin{itemize}
\item[$^a$] Department of Mathematics, 
University of Wisconsin -- Madison, 
480 Lincoln Drive,
Madison WI 53706, United States.
\item[$^b$] Korteweg-de Vries Institute (KdVI) for Mathematics,
University of Amsterdam, Science Park 904, 1098 XH Amsterdam, the Netherlands. 
\item[$^c$] CWI, Science Park 123, P.O. Box 94079, 1090 GB Amsterdam, the Netherlands.
\item[$^d$] TELIN, Ghent University, St.-Pietersnieuwstraat 41,
B9000 Gent, Belgium. 
\end{itemize}
\noindent {Work partially done while D.\ Anderson was visiting CWI and KdVI,
and  K.\ de Turck was visiting KdVI; financial support from 
NWO cluster {\it STAR} (Anderson) and {\it Fonds Wetenschappelijk Onderzoek / Research Foundation -- Flanders} (De Turck) is greatly appreciated. Anderson was also supported under NSF grant DMS-100975 and DMS-1318832.

\noindent
M.\ Mandjes is also with  E{\sc urandom}, Eindhoven University of Technology, Eindhoven, the Netherlands, 
and I{\sc bis}, Faculty of Economics and Business, University of Amsterdam,
Amsterdam, the Netherlands.}

\vb

\noindent $^\star$Corresponding author. Tel.:+31 205924167

\noindent {\it Email addresses}. {\scriptsize \tt{anderson@math.wisc.edu}, \tt{j.g.blom@cwi.nl}, \tt{m.r.h.mandjes@uva.nl}, 

\noindent \tt{h.thorsdottir@cwi.nl}, \tt{kdeturck@telin.ugent.be}}

\vb

\end{abstract}

\maketitle

\hyphenation{mo-lecules}

\newpage

\section{Introduction}
When modeling chemical reaction networks within cells, the dynamics of the numbers of molecules of the various types are often described by deterministic differential equations. These models ignore the inherent stochasticity that may play a role, particularly when the number of molecules are relatively small. 
To remedy this, the use of stochastic representations of chemical networks has been advocated, see e.g. \cite{AraziBenJacobYechiali2004,Cookson_etal2011,Gillespie2007}.

In this paper we use the formulation as in \cite{ANDE,BALL} where
the numbers of molecules evolve as a continuous-time Markov chain. A concise description of this formulation is the following, with our specific model more formally developed in Section \ref{sec:model}.
Consider a model consisting of a finite number, $\ell$, of species and a finite number, $K$, of reaction channels.  We let  $\vec{M}(t)$ be the $\ell$-dimensional vector whose $i$th component gives the number of molecules of the $i$th species present at time $t$.  For the $k$th reaction channel we denote by $\vec{\nu}_k \in {\mathbb Z}^\ell_{\ge 0}$  the number of molecules of each species needed for the reaction to occur, and by $\vec{\nu}'_k \in {\mathbb Z}_{\ge 0}^\ell$  the number produced.  We let $\mu_k(\vec{x})$ denote the rate, or intensity (termed \textit{propensity} in the biology literature), at which the $k$th reaction occurs when the numbers of molecules present equals the vector  $\vec{x}$.  Then, 
$\vec{M}(t)$ may be represented as the solution to the (vector-valued) equation
\begin{equation}\label{CN}
   \vec{M}(t) = \vec{M}(0) + \sum_{k=1}^K(\vec{\nu}'_k-\vec{\nu}_k) \,{Y}_k\left(\int_0^t \mu_k(\vec{M}(s)){\rm d}s\right), 
\end{equation}
where the stochastic processes $Y_k(\cdot)$ are independent unit-rate Poisson processes \cite{ANDE}.  Note that if,
for some $k^\star\in\{1,\ldots,K\}$,  $\vec{\nu}'_{k^\star}-\vec{\nu}_{k^
\star}$ equals the $i$th unit vector $\vec{e}_i$, then the $k^\star$th reaction channel corresponds to the external arrival of molecules of species $i$.
     For the specific situation that subnetworks operate at disparate timescales, these can be analyzed separately by means of lower dimensional approximations, as pointed out in e.g.\ \cite{KANG}.

%
%
%

\vb
 
 In this paper we study a model of the type described above, for the special case that there is just one type of molecular species (i.e., $\ell=1$), and that there are external arrivals. The distinguishing feature is that the rate of the external input  is determined by an independent continuous-time Markov chain $J(\cdot)$ (commonly referred to as the {\it background process}) defined on the finite state space $\{1,\ldots,d\}$. More concretely, we study a reaction system that obeys the stochastic representation
 \[M(t) = M(0) + Y_1\left(\int_0^t \lambda_{J(s)}{\rm d}s\right) -Y_2\left(\mu\int_0^t M(s){\rm d}s\right),
 \]
 where $Y_1(\cdot)$ and $Y_2(\cdot)$ are independent unit-rate Poisson processes, and $\lambda_{J(s)}$ takes the value $\lambda_i \ge0$ when the background process is in state $i$.  
 Hence, in this model external molecules flow into the system according to a Poisson process with rate $\lambda_i$ when the background process $J(\cdot)$ is in state $i$, while each molecule decays after an exponentially distributed time with mean $\mu^{-1}$ (independently of other molecules present).
 
 The main result of the paper is a functional central limit theorem ({\sc f-clt}) for the process $M(t)$, where we impose a specific scaling on the transition rates $Q=(q_{ij})_{i,j=1}^d$ of the background process $J(t)$, as well as on the external arrival rates $\vec{\lambda}=(\lambda_1,\ldots,\lambda_d)^{\rm T}$ (note that all vectors are to be understood as column vectors). More precisely, the transition rates of the background process are sped up by a factor $N^\alpha$, with $\alpha > 0$, while the arrival rates are sped up linearly, that is, they become $N\lambda_i$. 
 Then we consider the process $U^N(t)$ (with a superscript $N$ to stress the dependence on $N$), obtained from 
$M^N(t)$  by centering, that is, subtracting the mean ${\mathbb E}M^N(t)$, and normalizing, that is, dividing by an appropriate polynomial in $N$.
It is proven that $U^N(t)$  converges (as $N\to\infty$) weakly to a specific Gauss-Markov process, viz.\ an Ornstein-Uhlenbeck ({\sc ou}) process with certain parameters (which are given explicitly  in terms of $\vec{\lambda},\mu$, and the matrix $Q$). Our proofs are based on martingale techniques; more specifically, an important role is played by the martingale central limit theorem.

Interestingly, if $\alpha> 1$ the normalizing polynomial in the {\sc f-clt} is the usual $\sqrt{N}$, but for $\alpha\le 1$ it turns out that we have to divide by $N^{1-\alpha/2}$. The main intuition behind this dichotomy is that for $\alpha>1$ the timescale of the background process is faster than that of the arrival process, and hence the arrival process is effectively a (homogeneous) Poisson process. As a result the corresponding {\sc f-clt} is in terms of the corresponding Poisson rate (which we denote by $\lambda_\infty:=\vec{\pi}^{\rm T}\vec{\lambda}$, where $\vec{\pi}$ is the stationary distribution of the background process) and $\mu$ only. For $\alpha\le 1$, on the contrary, the background process jumps relatively slowly; the limiting {\sc ou} process is in terms of $\vec{\lambda}$ and $\mu$, but features the deviation matrix  \cite{COOL} associated to the background process $J(\cdot)$ as well.
 
 \vb
 
In earlier works \cite{BMT, KELL} we studied a similar setting. However, where we use a martingale-based approach in the present paper, in \cite{BMT, KELL} we relied on another technique: (i)~we set up a system of differential equations for the Laplace transform of $M^N(t)$ jointly with the state of the background process $J^N(t)$, (ii)~modified these into a system of differential equations for the transform of the (centered and normalized) process $U^N(t)$ jointly with $J^N(t)$, (iii)~approximated these by using Taylor expansions, and (iv)~then derived  an ordinary differential equation for the limit of the transform of $U^N(t)$ (as $N\to\infty$). This differential equation defining a Normal distribution, the {\sc clt} was established.
 Importantly, the results
 derived in \cite{BMT,KELL} crucially differ from the ones in the present paper. The most significant difference is that
 those results are {\it no} {\sc f-clt}:  just the {\it finite-dimensional convergence} to the {\sc ou} process was established, rather than convergence at the process level (i.e., to prove weak convergence an additional `tightness argument' would be needed). 
 For the sake of completeness, we mention that  \cite{KELL} covers just the case $\alpha>1$, that is, the regime in which the arrival process is effectively Poissonian, while \cite{BMT} allows all $\alpha>0$.
 
  Our previous works \cite{BMT,KELL,BLOM} have been presented in the language of {\it queueing theory}; the model described above can be seen as an infinite-server queue with  Markov-modulated input.  In comparison to  Markov-modulated single-server queues (and to a lesser extent 
 Markov-modulated many-server queues), this infinite-server model has been much less intensively studied.
 This is potentially due to the fact that the presence of infinitely many servers may be considered less realistic, perhaps rightfully so in the context of operational research, the major application field of queueing theory. In the context of chemical reactions, however, 
 it can be argued that the concept of infinitely many servers is quite natural: each molecule brings its own `decay'-server. 
 
 \vb

We conclude this introduction with a few short remarks on the relation of our work with existing literature. Incorporating Markov modulation in the external arrival rate the infinite-server queue becomes, from a biological perspective, a more realistic model \cite{BRUG}. It is noted that deterministic modulation has been studied in  \cite{EICK} for various types of non-homogeneous arrival rate functions  (M$_t$/G/$\infty$ queue). For earlier results on the stationary distribution of Markov-modulated infinite-server queues, for instance in terms of a recursive scheme that determines the moments, we refer to e.g.\ \cite{DAURIA,FRALIX,KEILSON,OCINNEIDE}. 

As mentioned above, at the methodological level, our work heavily relies on the so-called martingale central limit theorem ({\sc m-clt}), see for instance \cite{ETHI,WHITT}. It is noted that convergence to {\sc ou} has been established in the non-modulated setting before: an appropriately scaled M/M/$\infty$ queue weakly converges to an {\sc ou} process. For a proof, see 
e.g.\ \cite[Section 6.6]{ROBE}; cf.\ also \cite{BOR,IGL}. 
 
 \vb

The rest of this paper is organized as follows. Section \ref{sec:model} sets up the model, its properties and quantities of interest, and presents the essential mathematical tools. The $N^\alpha$-scaled background process is thoroughly investigated in Section \ref{sec:procX}; most notably we derive its {\sc f-clt} relying on the {\sc m-clt}. This takes us to Section \ref{sec:procM}, where we first show that $\bar{M}^N(t) := N^{-1}M^N(t)$ converges to a deterministic solution, denoted by  ${\varrho}(t)$, to finally establish the {\sc f-clt} for $M^N(t)$ by proving asymptotic normality of the process $N^\beta\,(\bar{M}^N(t)-\varrho(t))$, with $\beta\in(0,\sfrac{1}{2})$ appropriately chosen. As indicated earlier, the parameters specifying the limiting {\sc ou} process depend on which speedup is `faster': the one corresponding to  the background process (i.e., $\alpha>1$) or that of the arrival rates (i.e., $\alpha<1$). We conclude this paper by a set of numerical experiments, that illustrate the impact 
of the value of $\alpha.$

\section{The model and mathematical tools}\label{sec:model}
In this section we first describe our model in detail, and then present preliminaries (viz.\  a version of the law of large numbers for Poisson processes and the {\sc m-clt}).

\vb 

{\it Model.} Our paper considers the following Markovian model. Let $J(t)$ be an irreducible continuous-time Markov process on the finite state space $\{1,\ldots,d\}$. Define its generator matrix by $Q=(q_{ij})_{i,j=1}^d$ and the (necessarily unique) invariant distribution by $\vec{\pi}$; as a consequence, $\vec{\pi}^{\rm T}Q=\vec{0}^{\rm T}$. 
Let $X_i(t)$ be the indicator function of the event $\{J(t)=i\}$, for $i=1,\ldots,d$; in other words: $X_i(t)=1$ if $J(t)=i$ and $0$ otherwise. It is assumed that $J(\cdot)$ is in stationarity at time $0$ and hence at any $t$; we thus have ${\mathbb P}(J(t)=i)=\pi_i$.
As commonly done in the literature, the transient distribution ${\mathbb P}(J(s)=j\,|\,J(0)=i)$ is denoted by $p_{ij}(s)$ and is computed as $(e^{Qs})_{i,j}$.

The model considered in this paper is  a so-called {\it Markov-modulated infinite-server queue}. Its dynamics can be described as follows.
For any time $t \ge 0$, molecules arrive according to a Poisson process with rate $\lambda_i$ if $X_i(t)=1$. We let the service/decay rate of each molecule be  $\mu$ irrespective of the state of the background process. There are infinitely many servers so that the molecules' sojourn times are their service times; the molecules go in service immediately upon arrival. Throughout this paper, $M(t)$ denotes the number of molecules present at time $t$.

\vb

{\it Scaling.} In this paper a {\sc f-clt} under the  following scaling is investigated. The background process as well as the arrival process are sped up, while the service-time distribution remains unaffected. More specifically, the transition matrix of the background process becomes $N^\alpha Q$ for some $\alpha>0$, while the arrival rates, $\lambda_i$ for $i=1,\ldots,d$, are scaled linearly (i.e., become $N\lambda_i$); then $N$ is sent to $\infty$. To indicate the fact that they depend on the scaling parameter $N$, we write in the sequel $J^N(t)$ for the background process, $X_i^N(t)$ for the indicator function associated with state $i$ of the background process at time $t$, and $M^N(t)$ for the number of molecules in the system at time $t$. Later in the paper we let the transitions of the background process go from being sublinear (i.e., $\alpha<1$) to superlinear (i.e., $\alpha>1$); one of our main findings is that there is a {\it dichotomy}, in the sense that there is crucially different behavior in 
these two regimes, with a special situation at the boundary, i.e., $\alpha=1$.

\vb

The above model can be put in terms of a chemical reaction network, as formulated in the introduction. It turns out to be convenient to  do so by interpreting the background process as a model for a single molecule transitioning between $d$ different states, with $X_i^N(t)$ denoting the number of molecules in state $i$ at time $t$.  Since there is at most one such molecule, we see $X_i^N(t) \in\{0,1\}$.
The following table informally summarizes the relevant reactions  and corresponding intensity functions for the model of interest:\vb
\begin{align*}
	\begin{array}{|c|c|c|}\hline\hline
	\text{Reaction} & \:\:\:\text{Intensity function} \:\:\:&\text{Description}\\\hline
 	X_i \to X_j \:\:\mbox{(for $i\not=j$)}& N^{\alpha} q_{ij} X_i^N(t)&\text{$J(\cdot)$ jumps from $i$ to $j$}\\\hline
	\emptyset \to M & \sum_{i=1}^d N\lambda_i X_i^N(t)&\text{Arrival}\\\hline
	M \to \emptyset & \mu M^N(t)&\text{Departure}\\\hline\hline
	\end{array}
\end{align*} 
As mentioned above, it is assumed that $\alpha >0$; in addition,
$q_{ij}\ge 0$ for $i\not=j$ (while $Q\vec{1}=\vec{0}$), $\lambda_i\ge 0$, and $\mu>0$.
The dynamics can be phrased in terms of  the stochastic representation framework, as described in the introduction.
In the first place, the evolution of the indicator functions can be represented as
\begin{equation}
	X_i^N(t) \hmm= \hmm X_i^N(0) - \sum_{\substack{j=1\\j\neq i}}^dY_{i,j}\left( N^\alpha q_{ij}\int_0^t X_i^N(s) \DD s\right)+\sum_{\substack{j=1\\j\neq i}}^d Y_{j,i}\left( N^\alpha q_{ji}\int_0^t X_j^N(s) \DD s\right) \label{eqX1}
\end{equation}
	where the $Y_{i,j}$ ($i,j=1,\ldots,d$ with $i\not=j$) are independent unit-rate Poisson processes.
	It is readily verified that if the $X_i^N(0)$, with $i=1,\ldots,d$, are indicator functions summing up to $1$,
	then so are the $X_i^N(t)$ for any $t\ge 0.$ The second (third, respectively) term in the right-hand side represents the number of times that $J^N(\cdot)$ leaves (enters) state $i$ in $[0,t]$.
	
	In the second place, the number of molecules in the system evolves as
\begin{equation}	M^N(t)\hmm = \hmm M^N(0) + Y_1\left( N \int_0^t \sum_{i=1}^d \lambda_i X_i^N(s) \DD s\right) - Y_2\left( \mu \int_0^t M^N(s) \DD s\right),\label{eqM}
\end{equation}
where $Y_1,$ and $Y_2$ are independent unit-rate Poisson processes (also independent of the $Y_{i,j}$).

The objective of this paper is to describe the limiting behavior of the system as $N\to \infty$, for different values of $\alpha$. Our main result is a {\sc f-clt} for the process $M^N(\cdot)$; to establish this, we also need a {\sc f-clt} for 
the {\it state frequencies} of $J^N(\cdot)$ on $[0,t]$, defined as
 \[\vec{Z}^N(t)=(Z^N_1(t),\ldots, Z_d^N(t))^{\rm T},\quad\mbox{with}\quad
 Z^N_i(t):=\int_0^t X^N_i(s)\DD s.\] 
 
  \vb
  
 This paper essentially makes use of two more or less standard `tools' from probability theory: the  law of large numbers applied to Poisson processes and the martingale central limit theorem ({\sc m-clt}).  For the sake of completeness, we state the versions used here.
 
 \begin{lemma}\label{LLN} {\rm {\cite[Thm.\ 2.2]{ANDE}}}
 	Let $Y$ be a unit rate Poisson process.  Then for any $U > 0$,
	\[
		\lim_{N\to \infty} \sup_{0\le u\le U} \left| \frac{Y(Nu)}{N} - u\right| = 0,
	\]
	almost surely.
 \end{lemma}

The following is known as (a version of) the {\sc m-clt}, and is a corollary to Thm.\ 7.1.4 and the proof of Thm.\ 7.1.1 in \cite{ETHI}. Here and in the sequel, `$ \Rightarrow$' denotes weak convergence; in addition, $[\cdot,\cdot]_t$ is the quadratic covariation process.
\begin{theorem}\label{MCLT}
	Let $\{\vec{\mathcal{M}}^N\}$, for $N\in{\mathbb N}$, be a sequence of $\R^d$-valued martingales with $\vec{\mathcal{M}}^N(0) =\vec{0}$ for any $N\in{\mathbb N}$. Suppose
	\[
		\lim_{N\to \infty} \E \left[ \sup_{s\le t} \left|\vec{\mathcal{M}}^N(s) - \vec{\mathcal{M}}^N(s-)\right| \right] = \vec{0},
\text{ with } \vec{\mathcal{M}}^N(s-) := \lim_{u \uparrow s}\vec{\mathcal{M}}^N(u),
	\]
	and, as $N\to\infty$,
	\[
		[\mathcal{M}^N_i,\mathcal{M}^N_j]_t \to C_{ij}(t)
	\]
	for a deterministic matrix $C_{ij}(t)$ that is continuous in $t$, for $i,j=1,\ldots,d$ and for all $t>0$.
	Then $\vec{\mathcal{M}}^N \Rightarrow \vec{W}$, where $\vec{W}$ denotes a Gaussian process with independent increments and $\E\left[\vec{W}(t)\vec{W}(t)^{\rm T}\right] = C(t)$ $($such that $\E [W_i(t)\,W_j(t)] = C_{ij}(t)$$)$.
\end{theorem}
 
There is an extensive body of literature on the {\sc m-clt};  for more background, see e.g.\ \cite{JS,ROB,WHITT}.

 \section{A functional clt for the state frequencies}\label{sec:procX}
In this section we establish the {\sc f-clt} for the integrated background processes $\vec{Z}^N(t)$, that is, the state frequencies of the Markov process $J^N(\cdot)$ on $[0,t]$. This {\sc f-clt} is a crucial element in the proof of the {\sc f-clt}
of $M^N(t)$, as will be given in the next section. 
It is noted that there are several ways to establish this {\sc f-clt}; we refer for instance to the related weak convergence results in \cite{BHAT,KURTZ}, as well as the nice, compact proof for the single-dimensional convergence in \cite[Ch. II, Thm. 4.11]{ASMU}. We chose to include our own derivation, as it is straightforward, insightful and self-contained, while at the same time it also introduces a number of concepts and techniques that are used in the 
{\sc m-clt}-based proof of the {\sc f-clt}
for $M^N(t)$ in the next section.

We first identify the corresponding law of large numbers. To this end, we consider the process $\vec{X}^N(t)$ by dividing both sides of \eqref{eqX1} by $N^\alpha$ and letting $N\to \infty$.  Since $X_i^N(t)\in\{0,1\}$ for all $t$, the $X_i^N(t)$ and $X_i^N(0)$ terms both go to zero as $N\to\infty$. Thus we may apply Lemma \ref{LLN} to see that also, as $N\to\infty$,
 \begin{align*}
 	 - \sum_{j\neq i} q_{ij}Z_i^N(t)+ \sum_{j\neq i}q_{ji}Z_j^N(t) \to 0,
 \end{align*}
 almost surely, or $\lim_{N\to\infty}{\vec{Z}}^N(t)^{\rm T}\,Q=\vec{0}^{\rm T}$.
Bearing in mind that $\vec{1}^{\rm T}\vec{Z}^N(t)=t$, the limit of $\vec{Z}^N(t)$ solves the global balance equations, entailing that
\begin{equation}\label{eq:limitXint}
Z_j^N(t) \to \pi_j t
\end{equation}
 almost surely as $N\to\infty$, where we recall that $\vec{\pi}$ is the stationary distribution associated with the background process $J(\cdot)$.\\
 
\vb

As mentioned above, the primary objective of this section is to establish  a {\sc f-clt} for $\vec{Z}^N(\cdot)$ as $N\to\infty$. More specifically, we wish to identify a covariance matrix ${C}$ such that, as $N\to\infty$,
\begin{equation}\label{doel}
  N^{\alpha/2}\left(\vec{Z}^N(t) -\vec{\pi}t\right) \Rightarrow \vec{W}(t),
 \end{equation}
with $\vec{W}(\cdot)$ representing a ($d$-dimensional) Gaussian process with independent increments 
such that $\E\left[\vec{W}(t)\vec{W}(t)^{\rm T}\right] = C\,t$. In other words: our goal is to show weak convergence to a $d$-dimensional Brownian motion (with dependent components).

We start our exposition by identifying a candidate covariance matrix ${C}$, by studying the asymptotic behavior (that is, as $N\to\infty$) of $\COV(Z_i^N(t),Z_j^N(t))$ for fixed $t$. Bearing in mind that $Z^N_i(t)$ is the integral over $s$ of $X^N_i(s)$, and using standard properties of the covariance, this covariance can be rewritten as
\[\int_0^t \int_0^s \COV\left(X_i^N(r),X_j^N(s)\right)\,\DD r\DD s + 		
		\int_0^t \int_s^t \COV\left(X_i^N(r),X_j^N(s)\right)\,\DD r\DD s.\]
		Recalling that the process $J^N(\cdot)$ starts off in equilibrium at time $0$, and that $X_i(s)$ is the indicator function of the event $\{J^N(s)=i\}$, this expression can be rewritten as
	\[\int_0^t \int_0^s \left(\pi_i p_{ij}^N(s-r) - \pi_i \pi_j\right) \,\DD r\DD s+ 				\int_0^t \int_s^t \left(\pi_j p_{ji}^N(r-s) - \pi_i \pi_j\right) \,\DD r\DD s,\]
	where we use the notation $p_{ij}^N(s):={\mathbb P}(J^N(s)=j\,|\,J^N(0)=i).$	
Performing the change of variable $u:=rN^{\alpha}$ we thus find that
\newcommand{\hmmm}{\hspace{-0.5mm}}
\[N^\alpha \COV(Z_i^N(t),Z_j^N(t))=
 \pi_i \int_0^t \hmmm\int_0^{sN^{\alpha}} \hmmm\hmmm\left(p_{ij}(u) - \pi_j\right) \DD u\DD s + \pi_j \int_0^t \hmmm\int_0^{(t-s)N^{\alpha}} \hmmm\hmmm\left(p_{ji}(u) - \pi_i\right) \DD u\DD s.\]
 A crucial role in the analysis is played by  the {\it deviation matrix} $D=(D_{ij})_{i,j=1}^d$ associated with the finite-state Markov process $J(\cdot)$; it is defined by
 \begin{equation}
 D_{ij}:=\int_0^\infty (p_{ij}(t)-\pi_j)\DD t;\label{eq:deviation}
 \end{equation}
 see e.g.\  \cite{COOL} for background and a survey of the main results on deviation matrices. Combining the above, we conclude that, as $N\to\infty$, with $C_{ij}:= \pi_iD_{ij}+\pi_j D_{ji}$ we have identified that candidate covariance matrix, in the sense that we have shown that, for given $t$,
\[N^\alpha \,\COV(Z_i^N(t),Z_j^N(t)) \to C_{ij}t\] as $N\to\infty.$ 
The objective of the remainder of this section is to establish the weak convergence (\ref{doel}) with the covariance matrix $C=(C_{ij})_{i,j=1}^d$.

 \vb
We now prove this weak convergence relying on the {\sc m-clt}. We start by considering linear combinations
of the $X^N_i(t)$ processes based on Eqn.~\eqref{eqX1} and introduce $\tilde{X}_i^N(t):= X_i^N(t) - X_i^N(0)$ for notational convenience. For any real constants $f_i$, $i=1,\ldots,d$,
we have that
\begin{align}
    \sum_{i=1}^d f_i \tilde{X}^N_i(t) & = \sum_{i=1}^d\sum_{j\neq i} f_i \bigg(Y_{ji}\left(N^\alpha Z^N_j(t) q_{ji}\right) -Y_{ij}\left(N^\alpha Z^N_i(t) q_{ij}\right) \bigg) \nonumber \\
  & = \sum_{i=1}^d \sum_{j=1}^d (f_j - f_i) Y_{ij}\left(N^\alpha Z^N_i(t) q_{ij}\right), \label{eq:fX}
\end{align}
where we do not need to define the processes $Y_{ii}$ as the terms containing them are zero anyway. Note that the quadratic variation of this linear combination is equal to 
\begin{equation}
    \left[\sum_{i=1}^d f_i \tilde{X}^N_i\right]_t = \sum_{i=1}^d \sum_{j=1}^d (f_j - f_i)^2 Y_{ij}(N^\alpha Z^N_i(t) q_{ij}), \label{eq:quad-var}
\end{equation}
as the quadratic variation of a Poisson process is equal to itself.
Due to Lemma \ref{LLN} and Eqn.~\eqref{eq:limitXint}, we have for $N\to\infty$,
\begin{equation}
    \left[N^{-\alpha/2}\sum_{i=1}^d f_i \tilde{X}^N_i\right]_t \to t \sum_{i=1}^d \sum_{j=1}^d (f_j - f_i)^2 \pi_i q_{ij}.
\label{eq:quad-var-2}
\end{equation}

The crucial step in proving the weak convergence and consequently applying the {\sc m-clt} is the identification of a suitable martingale. 
We prove the following lemma.
\begin{lemma} Let $D$ denote the deviation matrix of the background Markov chain $J(t)$. $\vec V^N(t) := N^{-\alpha/2} \vec{\tilde{X}}^N(t)^T D + N^{\alpha/2}\left( \vec Z^N(t) - \boldsymbol\pi t\right)$ is
an $\mathbb R^d$-valued martingale.
\end{lemma}

\begin{proof}
    We center our unit-rate Poisson processes by introducing $\tilde Y_{i,j}(u) := Y_{i,j}(u) - u$. 
    The following algebraic manipulations are easily verified:
\begin{align*}
 & N^{-\alpha/2} \sum_{i=1}^d \sum_{j=1}^d (D_{jk} - D_{ik}) \tilde Y_{ij}\left(N^\alpha Z^N_i(t) q_{ij}\right)  \\
 & = N^{-\alpha/2} \sum_{i=1}^d \tilde{X}_i^N(t)D_{ik} - N^{\alpha/2} \sum_{i=1}^d \sum_{j=1}^d (D_{jk} - D_{ik}) Z^N_i(t) q_{ij} \\
 & =  N^{-\alpha/2}\left(\vec{\tilde{X}}^N(t)^T D\right)_k - N^{\alpha/2}\sum_{i=1}^d \sum_{j=1}^d Z^N_i(t) q_{ij} D_{jk} \\
 & =  N^{-\alpha/2}\left(\vec{\tilde{X}}^N(t)^T D\right)_k + N^{\alpha/2}(Z^N_k(t) - \pi_k t),
\end{align*}
where we used Eqn.~\eqref{eq:fX}, the fact that $\sum_{i=1}^d Z_i^N(t) = t$ and the property $QD = \Pi - I$, with $\Pi = \vec{1\pi}^{\rm T}$. As centered Poisson processes are martingales, and linear combinations preserve the martingale property, this concludes
the proof.
\end{proof}

We now wish to apply the {\sc m-clt} to $\vec V^N(t)$ as $N\to\infty$. As the second term is absolutely continuous,
and the first term is a jump process with jump sizes $N^{-\alpha/2}$, we have indeed vanishing jump
sizes as required by the {\sc m-clt}. We now compute the covariations of $\vec V^N(t)$, and note that
as the second term is absolutely continuous and thus does not contribute to the covariation, we have that
\begin{align}
    [V^N_i,V^N_j]_t & = N^{-\alpha} [((\vec{\tilde{X}}^N)^TD)_i, ((\vec{\tilde{X}}^N)^TD)_j]_t \nonumber\\
  & = \frac12  N^{-\alpha} \left( [((\vec{\tilde{X}}^N)^TD)_i + ((\vec{\tilde{X}}^N)^TD)_j]_t - [((\vec{\tilde{X}}^N)^TD)_i]_t - [((\vec{\tilde{X}}^N)^TD)_j]_t \right) \nonumber\\
  & = \frac12  N^{-\alpha} \left( \left[\sum_{k=1}^d \tilde{X}_k^N (D_{ki} + D_{kj})\right]_t -\left[\sum_{k=1}^d \tilde{X}_k^N D_{ki}\right]_t - \left[\sum_{k=1}^d \tilde{X}_k^N D_{kj}\right]_t \right),
\end{align}
where we have used the polarization identity $2[X,Y]_t = [X+Y]_t - [X]_t - [Y]_t$.

Using Eqn.~\eqref{eq:quad-var-2}, this converges to
\begin{align}
    [V^N_i,V^N_j]_t & \to \frac12 t \sum_{k=1}^d \sum_{\ell=1}^d \pi_k q_{k\ell} \bigg( (D_{ki} + D_{kj} -D_{\ell i} - D_{\ell j})^2  - (D_{ki} - D_{\ell i})^2 - (D_{kj} - D_{\ell j})^2\bigg)\nonumber\\
& = t \sum_{k=1}^d \sum_{\ell=1}^d \pi_k q_{k\ell} (D_{ki} - D_{\ell i})(D_{kj} - D_{\ell j}) \nonumber\\
& = t (\pi_j D_{ji} + \pi_i D_{ij})
\end{align}
where we used the properties $Q\vec{1}=0$, $\Pi Q=0$, $QD = \Pi-I$ and $\Pi D = 0$.

Thus, from the {\sc m-clt} we have that $\vec V^N(t)$ converges weakly to $d$-dimensional Brownian motion with covariance matrix $C:=
D^{\rm T}{\rm diag}\{\vec{\pi}\}+{\rm diag}\{\vec{\pi}\}D$.
As the first term of $\vec V^N(t)$ vanishes for $N\to\infty$, we have established the desired {\sc f-clt}:
\begin{proposition}\label{prop:zNormal} As $N\to\infty$,\[
 N^{\alpha/2}\left(\vec{Z}^N(t)-\vec{\pi}t\right)  \Rightarrow \vec{W}_{C}(t),
\]
where $\vec{W}_{C}(\cdot)$ is a zero-mean Gaussian process with independent increments and
covariance structure
$\E\left[\vec{W}_{C}(t)\vec{W}_{C}(t)^{\rm T}\right] = C\,t$. 
\end{proposition}

\section{A functional clt for the process $M^N(t)$}\label{sec:procM}
Using the {\sc f-clt} for the process ${\vec{Z}}^N(t)$, as established in the previous  section, 
we are now in a position to understand the limiting behavior of the main process of  interest, $M^N(t)$,
as $N$ grows large.  As before, we begin by considering the average behavior of the quantity of interest. Dividing both sides of \eqref{eqM} by $N$, and denoting $\bar  M^N(t) := N^{-1}M^N(t)$, we have
\begin{equation*}
	\bar  M^N(t) = \bar  M^N(0) + N^{-1}Y_1\left( N \sum_{i=1}^d\lambda_i Z_i^N(t)\right) - N^{-1} Y_2\left( N \mu \int_0^t \bar  M^N(s) ds\right).
\end{equation*}
Assuming that $\bar  M^N(0)$ converges almost surely to some value ${\varrho}_0$,  the use of Lemma \ref{LLN} in conjunction with Eqn.\ \eqref{eq:limitXint} yields that $\bar  M^N(t)$ converges almost surely to the solution of the deterministic integral equation
\begin{equation}
	{\varrho}(t) = {\varrho}_0 + \left(\sum_{i=1}^d\lambda_i\pi_i\right) t  - \mu \int_0^t {\varrho}(s)\DD s={\varrho}_0 + \lambda_\infty t - \mu\int_0^t {\varrho}(s) \DD s,\label{eq:touse_rho}
\end{equation}
with $\lambda_\infty:=\vec{\pi}^{\rm T}\vec{\lambda}$.
It is readily verified that this solution is given by
\begin{equation}\label{eq:rhot}
	{\varrho}(t) = {\varrho}_0e^{-\mu t} + \frac{\lambda_\infty}{\mu} (1 - e^{-\mu t}).
\end{equation}


As our goal is to derive a {\sc f-clt}, we center and scale the process $M^N(\cdot)$; this we do by subtracting
$N\varrho(\cdot)$, and dividing by $N^{1-\beta}$ for some $\beta>0$ to be specified later.
More concretely, we introduce the process
\[U^N_\beta(t) := N^\beta \big( \bar  M^N(t)  - {\varrho}(t) \big).\]
Letting $\beta >0$ be arbitrary (for the moment), we have that due to Eqn.\ \eqref{eq:touse_rho},
\begin{eqnarray*}
	N^\beta \left( \bar  M^N(t)  - {\varrho}(t) \right)
	&=& N^\beta\left(\bar  M^N(0) - \varrho_0\right) - N^\beta({\varrho}(t) - \varrho_0) \\
	&&\hspace{.02in}+ \,N^\beta \bigg( N^{-1}Y_1\left( N \sum_{i=1}^d\lambda_i Z_i^N(t)\right) - N^{-1} Y_2\left( N \mu \int_0^t \bar  M^N(s) \DD s\right)\bigg)\\[2ex]
		&= &N^\beta(\bar  M^N(0) - \varrho_0)-N^\beta\left(\lambda_\infty t - \mu \int_0^t {\varrho}(s) \DD s\right) \\
	&&\hspace{.02in}+ \,N^\beta \bigg( N^{-1}\tilde Y_1\left( N \sum_{i=1}^d\lambda_i Z_i^N(t)\right) - N^{-1} \tilde Y_2\left( N \mu \int_0^t \bar  M^N(s) \DD s\right)\bigg)\\
	&&\hspace{.02in} + \,N^{\beta} \sum_{i=1}^d\lambda_i Z_i^N(t) - N^\beta \mu \int_0^t \bar  M^N(s) \DD s.
\end{eqnarray*}
This identity can be written in a more convenient form by defining the process
\begin{align*}
	R_\beta(t) &:=  \tilde Y_1\left( N \sum_{i=1}^d\lambda_i Z_i^N(t)\right) -  \tilde Y_2\left( N \mu \int_0^t \bar  M^N(s) \DD s\right),
\end{align*}
which  is a martingale \cite{ANDE}. The resulting equation for $U_\beta^N(t)$ is
\begin{align}\label{eq:main_eqU}
	U^N_\beta(t) &= U^N_\beta(0) + N^{\beta-1}R_\beta(t) + N^\beta \left(\sum_{i=1}^d\lambda_i Z_i^N(t) - \lambda_\infty t\right) - \mu \int_0^t U^N_\beta(s) \DD s.
\end{align}
 We wish to establish the weak convergence of the process $U_\beta^N(t)$, as $N\to\infty$. We must simultaneously consider how to choose the parameter $\beta$. To do so we separately inspect the  terms involving $R_\beta(t)$ and $Z_i^N(t)$ in Eqn.\ \eqref{eq:main_eqU}.
 
First note that the sequence of martingales $\{N^{\beta-1}R_\beta(t)\}$, for $N\in{\mathbb N}$, clearly satisfies the first condition of Thm.\ \ref{MCLT}, that of vanishing jump sizes, under the condition that $\beta<1$, which we impose from now on. To obtain its weak limit, we compute its quadratic variation

\begin{equation}\label{eq:RbetaVar}
	\left[N^{\beta-1}R_\beta\right]_t = N^{2\beta-2}\left(Y_1\left( N \sum_{i=1}^d\lambda_i Z_i^N(t)\right) + Y_2\left( N \mu \int_0^t \bar  M^N(s) \DD s\right) \right).
\end{equation}
For this term to converge in accordance with Thm.\ \ref{MCLT}, we need $\beta \leq \frac{1}{2}$, which we impose from now on. With $\beta = \frac{1}{2}$, the term \eqref{eq:RbetaVar} will converge to $\lambda_\infty t + \mu\int_0^t {\varrho}(s) \DD s$. Choosing $\beta < \frac{1}{2}$ will take it to zero. Turning to the $Z_i^N$ terms of \eqref{eq:main_eqU}, by Prop.\ \ref{prop:zNormal}, and recalling that $\lambda_\infty=\vec{\pi}^{\rm T}\vec{\lambda}$,  we have that for $\beta = \alpha/2$,
\begin{align}\label{eq:lamZ}
N^\beta\left(\sum_{i=1}^d\lambda_i Z_i^N(t) - \lambda_\infty t\right) \Rightarrow \vec{\lambda} \cdot  \vec{W}_{C}(t),
\end{align}
which is distributionally equivalent to $W(\mbox{{\rm \TH}}t)$, where $W$ is a standard Brownian motion and
\begin{equation}\label{eq:thorn}
	{\mbox{{\rm \TH}}}:=\sum_{i=1}^d\sum_{j=1}^d\lambda_i\lambda_jC_{ij}. 
\end{equation}

If $\beta < \alpha/2$ the term on the left-hand side of \eqref{eq:lamZ} converges to zero.
Combining the above leads us to select $\beta = \min\{\alpha/2,1/2\}$.  In the sequel we distinguish between $\alpha > 1$, $\alpha < 1$, and $\alpha = 1.$

\vb

Before we treat the three cases, we first recapitulate the class of Ornstein-Uhlenbeck ({\sc ou}) processes. We say that $S(t)$ is {\sc ou}$(a,b,c)$ if it satisfies the stochastic differential equation 
({\sc sde}) 
\[\DD S(t) = (a-b\,S(t))\DD t + \sqrt{c}\,\DD W(t),\]
with $W(t)$ a standard Brownian motion. This {\sc sde} is solved by
\[S(t) = S(0) e^{-bt} +a\int_0^t e^{-b(t-s)}\DD s + \sqrt{c}\int_0^t e^{-b(t-s)}\DD W(s).\]
By using standard stochastic calculus it can be verified that (taking $u\le t$)
\begin{eqnarray}
{\mathbb E}S(t)&=&S(0)e^{-bt}+\frac{a}{b}(1-e^{-bt}),\\
\nonumber\VAR \,S(t) &=& \frac{c}{2b}(1-e^{-2bt}),\\
\nonumber\COV(S(t),S(u)) &=& \frac{ce^{-bu}}{2b}
\left(e^{bt}\hspace{-1mm}-e^{-bt}\right).
\end{eqnarray}\label{eq:OUprop}
For $t$ large, we see that
\[{\mathbb E}S(\infty)=\frac{a}{b},\quad\VAR \,S(\infty) = \frac{c}{2b},\quad\lim_{t\to\infty}\COV(S(t),S(t+u) )=  \frac{c}{2b}e^{-bu}
.\]
After this intermezzo, we now treat the three cases separately.

\subsection*{Case 1: $\alpha>1$} In this case we pick $\beta=1/2$. The term \eqref{eq:RbetaVar} converges  to
\[
\sum_{i=1}^d \lambda_i\pi_i t + \mu\int_0^t {\varrho}(s)\DD s = \lambda_{\infty}t + \mu\int_0^t{\varrho}(s)\DD s,
\]
while the term \eqref{eq:lamZ} converges to zero and is therefore neglected.
Hence, $U_{1/2}^N(t)$ converges in  distribution to the solution of 
\[
U_{1/2}(t) = U_{1/2}(0) + W\left(\lambda_{\infty}t + \mu\int_0^t{\varrho}(s)\DD s\right) - \mu\int_0^t U_{1/2}(s)\DD s,
\]
where $W$ is a standard Brownian motion. The above solution is distributionally equivalent to the solution of the It\^o formulation of the {\sc sde} 
\[
U_{1/2}(t) = U_{1/2}(0) + \int_0^t \sqrt{\lambda_{\infty} + \mu{\varrho}(s)}\DD W(s) - \mu\int_0^tU_{1/2}(s)\DD s.
\]
This {\sc sde} can be solved using standard techniques to obtain
\[
U_{1/2}(t) = e^{-\mu t}\left(U_{1/2}(0) + \int_0^t \sqrt{\lambda_{\infty} + \mu{\varrho}(s)}e^{\mu s} \DD W(s)\right)\hmmm.
\]
We now demonstrate how to compute the variance of 
$U_{1/2}(t)$. To this end, first recall that by virtue of \eqref{eq:rhot},
\begin{eqnarray*}
\VAR\, U_{1/2}(t) &= &\int_0^t (\lambda_{\infty} + \mu{\varrho}(s))e^{-2\mu(t-s)}\DD s \\
&=& \int_0^t \left(\lambda_{\infty} + \mu\left({\varrho}_0e^{-\mu s} + \frac{\lambda_{\infty}}{\mu}(1-e^{-\mu s})\right)\right)e^{-2\mu(t-s)}\DD s.
\end{eqnarray*}
After routine calculations, this yields
\[\VAR \,U_{1/2}(t) =\left( {\varrho}_0e^{-\mu t} + \frac{\lambda_{\infty}}{\mu}\right)(1 - e^{-\mu t}),
\]
cf.\ the expressions in \cite[Section 4]{KELL}.
In a similar fashion, we can derive that
\[\COV(U_{1/2}(t),U_{1/2}(t+u)) = e^{-\mu u}\left( {\varrho}_0e^{-\mu t} + \frac{\lambda_{\infty}}{\mu}\right)(1 - e^{-\mu t}).\]
It is seen that for $t\to\infty$ the limiting process behaves as {\sc ou}$(0,\mu,2\lambda_\infty)$. 

\subsection*{Case 2: $\alpha<1$}  In this case we pick $\beta = \alpha/2$ and the 
term~\eqref{eq:RbetaVar}, and therefore the term $N^{\beta-1}R_\beta(t)$, converges to zero, whereas the term \eqref{eq:lamZ} converges to
$W(\mbox{{\rm \TH}}t)$, where $W$ is a standard Brownian motion and \TH\ as defined by Eqn.\ \eqref{eq:thorn}.
Hence, $U_{\alpha/2}^N(t)$ converges weakly to the solution of
\[
U_{\alpha/2}(t) = U_{\alpha/2}(0) + W(\mbox{{\rm \TH}}t) - \mu\int_0^t U_{\alpha/2}(s)\DD s, \]

It is straightforward to solve this equation:
\[
U_{\alpha/2}(t) = e^{-\mu t}\left(U_{\alpha/2}(0) + \int_0^t \sqrt{{\mbox{{\rm \TH}}}}\,e^{\mu s}\DD W(s)\right).
\]
This process has variance
\begin{equation}
\VAR\,U_{\alpha/2}(t) = \int_0^t {\mbox{{\rm \TH}}} e^{-2\mu(t-s)} \DD s
= {\mbox{{\rm \TH}}} \frac{1-e^{-2\mu t}}{2\mu}.\label{eq:varalphaless}
\end{equation}
It is readily checked that this process is {\sc ou}$(0, \mu,{\mbox{{\rm \TH}}})$; this is due to
\[\COV(U_{\alpha/2}(t),U_{\alpha/2}(t+u)) ={\mbox{{\rm \TH}}} e^{-\mu u}\frac{1-e^{-2\mu t}}{2\mu}.\]

\subsection*{Case 3: $\alpha=1$}  In this case we put $\beta = 1/2$, and the terms 
$N^{\beta-1}R_\beta(t)$ and \eqref{eq:lamZ} are of the same order. Hence,  their sum
converges weakly to
\[
W\left(\lambda_{\infty}t + \mu\int_0^t {\varrho}(s)ds + {\mbox{{\rm \TH}}} t\right),
\]
where $W$ is a standard Browian motion.
In this case $U_{1/2}^N(t)$ converges weakly to the solution of 
\begin{align*}
U_{1/2}(t) 
&= U_{1/2}(0) + W\left(\int_0^t\left(\lambda_{\infty} + {\mbox{{\rm \TH}}} + \mu {\varrho}(s)\right)\DD s\right) - \mu\int_0^t U_{1/2}(s)\DD s.
\end{align*}
Solving the above in a similar fashion to cases 1 and 2 yields 
\[
U_{1/2}(t) = e^{-\mu t}\left(U_{1/2}(0) + \int_0^t \sqrt{\lambda_{\infty} + {\mbox{{\rm \TH}}} + \mu{\varrho}(s)}\,e^{\mu s} \DD W(s)\right).
\]
with the corresponding variance
\begin{align*}
\VAR \,U_{1/2}(t) &= \int_0^t \left(\lambda_{\infty} + {\mbox{{\rm \TH}}}+ \mu{\varrho}(s)\right)e^{-2\mu(t-s)}\DD s \\
&= \left({\varrho}_0e^{-\mu t} + \frac{\lambda_{\infty}}{\mu}\right)(1 - e^{-\mu t}) + \frac{{\mbox{{\rm \TH}}}}{2\mu}(1-e^{-2\mu t})
\end{align*}
and covariance
\[\COV(U_{1/2}(t),U_{1/2}(t+u)) = e^{-\mu u}\left( \left({\varrho}_0e^{-\mu t} + \frac{\lambda_{\infty}}{\mu}\right)(1 - e^{-\mu t}) + \frac{{\mbox{{\rm \TH}}}}{2\mu}(1-e^{-2\mu t})\right)\hmmm.\]
For $t$ large this process behaves as {\sc ou}$(0, \mu,2\lambda_\infty+{\mbox{{\rm \TH}}})$.

\vb

We summarize the above results in the following theorem; it is the {\sc f-clt} for $M^N(t)$ that we wished to establish.
It identifies the Gauss-Markov process to which $M^N(\cdot)$ weakly converges, after centering and scaling; this limiting process behaves, modulo the effect of the initial value $\varrho_0$, as an {\sc ou} process.
More specifically, the theorem describes the limiting behavior of the centered and normalized version $U^N_\beta(\cdot)$ of $M^N(\cdot)$: the focus is on the process
\begin{equation}\label{eq:U}
U_\beta^N(t) =N^\beta\left(\bar M^N(t)-\varrho(t)\right) =\frac{M^N(t)-N\varrho(t)}{N^{1-\beta}}.
\end{equation}
It is observed that for $\alpha\ge 1$, we have the usual $\sqrt{N}$ {\sc clt}-scaling;
for $\alpha<1$, however, the normalizing polynomial is $N^{1-\alpha/2}$, that is  $\beta=\min\{\alpha/2,1/2\}$.

\begin{theorem}\label{ST}
 As $N\to\infty$, the process $U_\beta^N(t)$ converges in distribution to the solution of
  \[
  U_\beta(t) = e^{-\mu t}\left(U_\beta(0) + \int_0^t \sigma(s) e^{\mu s} \,\DD W(s)\right)
 \]
where 
\[ \sigma(s) := 
\begin{cases}
  \sqrt{\lambda_{\infty} + \mu{\varrho}(s)}, & \alpha>1,\ \beta=1/2; \\
  \sqrt{\,{\mbox{{\rm \TH}}}}, & \alpha<1,\ \beta=\alpha/2; \\
  \sqrt{\lambda_{\infty} + {\mbox{{\rm \TH}}}+ \mu{\varrho}(s)}, & \alpha=1,\ \beta=1/2, \\
\end{cases}
\]

$W$ is standard Brownian motion and ${\mbox{{\rm \TH}}}=\sum_{i=1}^d\sum_{j=1}^d\lambda_i\lambda_jC_{ij}$. 
\end{theorem}

\vb
\section{Discussion and an Example}
Above we identified two crucially different scaling regimes: $\alpha>1$ and $\alpha<1$
(where the boundary case of $\alpha=1$ had to be dealt with separately). In case $\alpha>1$, the background process evolves fast relative to the arrival process, and as a consequence  the arrival stream is effectively Poisson with rate $N\lambda_\infty$. When the arrival process is simplified in such a way, the system essentially behaves as an M/M/$\infty$ queue. This regime was discussed in greater detail in e.g.\ \cite{KELL}, focusing on convergence of the finite-dimensional distributions.
On the other hand, for $\alpha < 1$ the arrival rate is sped up more than the background process. Intuitively, then the  system {settles} in a temporary (or  local) equilibrium. 
%

\subsection{A two-state example}
In this example we numerically study the limiting behavior of $U^N_\beta(t):=N^\beta\,(\bar M^N(t) - \varrho(t))$ with $\beta=\min\{\alpha/2,1/2\}$ (as $N\to\infty$) in a two-state system for different $\alpha$. For various values of $N$, we compute the moment generating function ({\sc mgf}) of $U^N_\beta(t)$ by numerically evaluating the system of differential equations
derived in \cite[Section 3.1]{KELL}.
We have shown that the limiting distribution of $U^N_\beta(t)$ is Gaussian with specific parameters. We now explain how
the {\sc mgf} of the limiting random variable can be computed. Introducing the notation $\Lambda(t,\theta) := {\mathbb E} e^{\theta U_\beta(t)}$ and $\Lambda^N(t,\theta) := {\mathbb E} e^{\theta U^N_\beta(t)}$, it is immediate from Thm.\ \ref{ST} that we have
\begin{equation}\label{eq:limitalphalarger}
\Lambda(t,\theta)=\exp\left(\frac{\theta^2}{2} \left[\left( {\varrho}_0e^{-\mu t} + \frac{\lambda_{\infty}}{\mu}\right)(1 - e^{-\mu t}) 
\,1_{\{\alpha\ge 1\}}+ {\mbox{{\rm \TH}}} \frac{(1-e^{-2\mu t})}{2\mu}
\,1_{\{\alpha\le 1\}}\right]\right),
\end{equation}
In the regime that $\alpha \le 1$, we need to evaluate the parameter ${\mbox{{\rm \TH}}} $, which can be easily computed for $d=2$.
For the generator matrix $Q=(q_{ij})_{i,j=1}^2$, let $q_i := -q_{ii}$ and note that $q_i>0$.
With $\bar{q} := q_1+q_2$, the matrix exponential is given by
\[
 e^{Qt} = \frac{1}{\bar{q}} \left[ \begin{array}{cc}
q_2 + q_1 e^{-\bar{q}t} & q_1 - q_1 e^{-\bar{q}t} \\
q_2 - q_2 e^{-\bar{q}t} & q_1 + q_2e^{-\bar{q}t} \end{array} \right].
\]
Since $\pi_1 = q_2/\bar{q}$ and $\pi_2 = q_1/\bar{q}$, we can now compute the components of the deviation matrix $D$ (see Eqn.\ \eqref{eq:deviation}) and covariance matrix $C$:
\begin{equation*}
 D 
  = \frac{1}{\bar{q}^2}
  \left[ \begin{array}{cc}
          q_1 & q_1 \\
          -q_2 & q_2
         \end{array}\right] \:\:\mbox{,
 }\:\: C = \frac{2q_1 q_2}{\bar{q}^3}\left[ \begin{array}{cc}
1 & -1 \\
-1 & 1 \end{array} \right].
\end{equation*}
From the above we find the value of ${\mbox{{\rm \TH}}}$:
\[
{\mbox{{\rm \TH}}} = \frac{2q_1 q_2}{\bar{q}^3}(\lambda_1 - \lambda_2)^2.
\]

Fig.\ \ref{fig:alphas} illustrates the convergence of $U^N_\beta(t)$ to $U_\beta(t)$ when $\vec{q}=(1,3)$, $\vec{\lambda} = (1,4)$ and $\mu=1$. We assume $\varrho_0 = 0$ and let $\theta=0.5$. In Figs.\ \ref{fig:alphas}.({\sc a})--({\sc c}) we see the effect of $\alpha$. Fig.\ \ref{fig:alphas}.({\sc d}) depicts the convergence rate, computed as
\[
 \max_{t\ge 0} \left|\Lambda^N(t,\theta) - \Lambda(t,\theta) \right|.
\]
We observe a roughly loglinear convergence speed for $\alpha\geq 1$, whereas for $\alpha <1$ the convergence turns out to be substantially slower.


\begin{figure}
        \centering
        \begin{subfigure}[b]{0.6\textwidth}
                \centering
                \includegraphics[scale=0.6, trim=3cm 8cm 0cm 6cm]{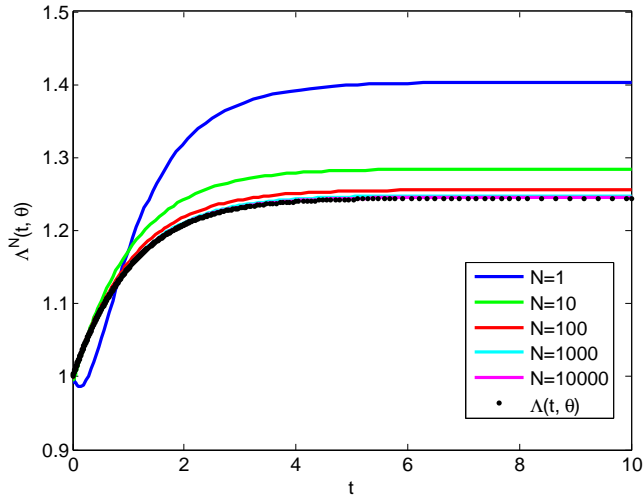}
                \caption{$\alpha=1.5$}
                \label{fig:alpha15}
        \end{subfigure}%
        ~ 
        \begin{subfigure}[b]{0.6\textwidth}
                \centering
                \includegraphics[scale=0.6, trim=6cm 8cm 4cm 6cm]{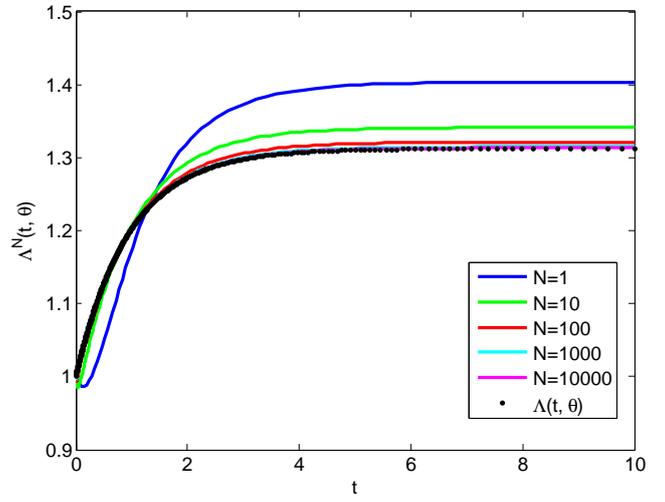}
                \caption{$\alpha=1$}
                \label{fig:alpha1}
        \end{subfigure}
        
        ~ 
        \begin{subfigure}[b]{0.6\textwidth}
                \centering
                \includegraphics[scale=0.6, trim=3cm 8cm 0cm 6cm]{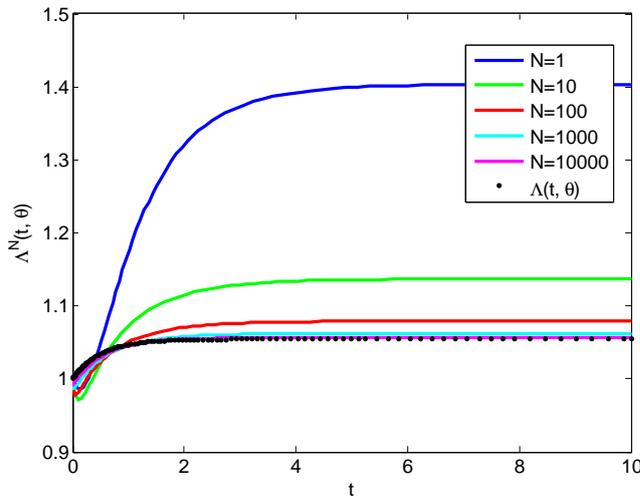}
                \caption{$\alpha=0.5$}
                \label{fig:alpha05}
        \end{subfigure}
        ~
        \begin{subfigure}[b]{0.6\textwidth}
                \centering
                \includegraphics[scale=0.6, trim=6cm 8cm 4cm 6cm]{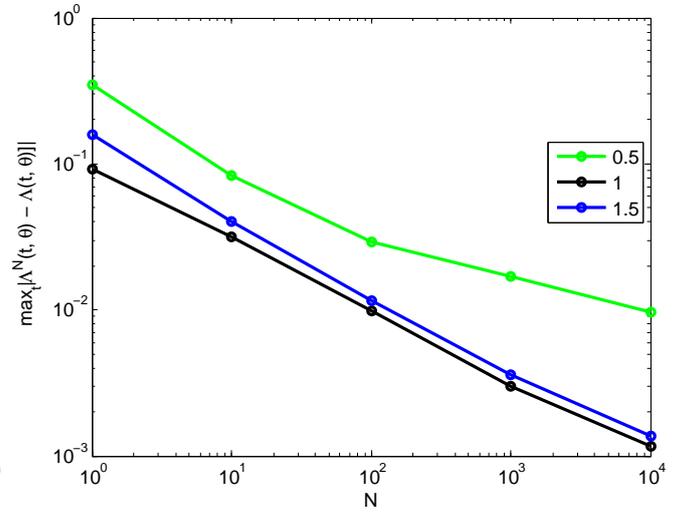}
                \caption{Error with $\alpha \in \{0.5,1,1.5\}$.}
                \label{fig:error}
        \end{subfigure}
	        \caption{Convergence of the {\sc mgf} of $U_\beta^N(t)$ as a function of $t\ge 0$. Panels ({\sc a})-({\sc c}) correspond to three values of $\alpha$; panel ({\sc d}) depicts the convergence rate as a function of $N$.\label{fig:alphas}}
\end{figure}

\begin{acknowledgments} We wish to thank T.\ Kurtz (University of Wisconsin -- Madison) for pointing us at several useful references.
\end{acknowledgments}

\bibliographystyle{plain}
\bibliography{bibdDim}

\begin{thebibliography}{10}

\bibitem{ANDE}
D.~Anderson and T.~Kurtz.
\newblock Continuous-time {M}arkov chain models for chemical reaction networks.
\newblock In H.~Koeppl et~al, editor, {\em Design and Analysis of Biomolecular
  Circuits: Engineering Approaches to Systems and Synthetic Biology}, pages
  3--42. Springer, New York, 2011.

\bibitem{AraziBenJacobYechiali2004}
A.~Arazia, E.~Ben-Jacob, and U.~Yechiali.
\newblock Bridging genetic networks and queueing theory.
\newblock {\em Physica A}, 332:585--616, 2004.

\bibitem{ASMU}
S.~Asmussen.
\newblock {\em Applied Probability and Queues, 2nd ed}.
\newblock Springer, New York, 2003.

\bibitem{BALL}
K.~Ball, T.~Kurtz, L.~Popovic, and G.~Rempala.
\newblock Asymptotic analysis of multi-scale approximations to reaction
  networks.
\newblock {\em Ann. Appl. Probab}, 16:1925--1961, 2006.

\bibitem{BHAT}
R.~N. Bhattacharya.
\newblock On the functional central limit theorem and the law of the iterated
  logarithm for {M}arkov processes.
\newblock {\em Z. Wahrscheinlichkeitstheorie verw. Gebiete}, 60:185--201, 1982.

\bibitem{BMT}
J.~Blom, K.~de~Turck, and M.~Mandjes.
\newblock A central limit theorem for {M}arkov-modulated infinite-server
  queues.
\newblock {\em In: {\it Proceedings ASMTA 2013}, Ghent, Belgium. Lecture Notes
  in Computer Science (LNCS) Series}, 7984:81--95, 2013.

\bibitem{KELL}
J.~Blom, O.~Kella, M.~Mandjes, and H.~Thorsdottir.
\newblock Markov-modulated infinite-server queues with general service times.
\newblock {\em To appear in {\it Queueing Systems} (available online)}, 2013.

\bibitem{BLOM}
J.~Blom, M.~Mandjes, and H.~Thorsdottir.
\newblock Time-scaling limits for markov-modulated infinite-server queues.
\newblock {\em Stoch. Models}, 29:112--127, 2013.

\bibitem{BOR}
A.~Borovkov.
\newblock On limit laws for service processes in multi-channel systems.
\newblock {\em Siberian Math. J}, 8:746--763, 1967.

\bibitem{Cookson_etal2011}
N.~Cookson, W.~Mather, T.~Danino, O.~Mondrag\'on-Palomino, R.~Williams,
  L.~Tsimring, and J.~Hasty.
\newblock Queueing up for enzymatic processing: Correlated signaling through
  coupled degradation.
\newblock {\em Mol. Syst. Biol}, 7(561), 2011.

\bibitem{COOL}
P.~Coolen-Schrijner and E.~van Doorn.
\newblock The deviation matrix of a continuous-time {M}arkov chain.
\newblock {\em Probab. Engrg. Inform. Sci}, 16:351--366, 2002.

\bibitem{DAURIA}
B.~D'Auria.
\newblock M/{M}/$\infty$ queues in semi-markovian random environment.
\newblock {\em Queueing Syst}, 58:221--237, 2008.

\bibitem{EICK}
S.~Eick, W.~Massey, and W.~Whitt.
\newblock The physics of the {M}$_t$/{G}/$\infty$ queue.
\newblock {\em Oper. Res}, 41:731--742, 1993.

\bibitem{ETHI}
S.~Ethier and T.~Kurtz.
\newblock {\em {M}arkov Processes: Characterization and Convergence}.
\newblock John Wiley \& Sons, 1986.

\bibitem{FRALIX}
B.~Fralix and I.~Adan.
\newblock An infinite-server queue influenced by a semi-{M}arkovian
  environment.
\newblock {\em Queueing Syst}, 61:65--84, 2009.

\bibitem{Gillespie2007}
D.~Gillespie.
\newblock Stochastic simulation of chemical kinetics.
\newblock {\em Annu. Rev. Phys. Chem}, 58:35--55, 2007.

\bibitem{IGL}
D.~Iglehart.
\newblock Limiting diffusion approximations for the many server queue and the
  repairman problem.
\newblock {\em J. Appl. Probab}, 2:429--441, 1965.

\bibitem{JS}
J.~Jacod and A.~Shiryayev.
\newblock {\em Limit Theorems for Stochastic Processes}.
\newblock Springer, Berlin, 1987.

\bibitem{KANG}
H.~Kang and T.~Kurtz.
\newblock Separation of time-scales and model reduction for stochastic reaction
  networks.
\newblock {\em Ann. Appl. Probab}, 23:529--583, 2013.

\bibitem{KEILSON}
J.~Keilson and L.~Servi.
\newblock The matrix {M}/{M}/$\infty$ system: retrial models and {M}arkov
  modulated sources.
\newblock {\em Adv. Appl. Probab}, 25:453--471, 1993.

\bibitem{KURTZ}
T.~Kurtz and P.~Protter.
\newblock Wong-zakai corrections, random evolutions, and simulation schemes for
  {SDE}s.
\newblock In E.~Merzbach E.~Mayer-Wolf and A.~Schwartz, editors, {\em
  Stochastic Analysis}, pages 331--346. Academic Press, 1991.

\bibitem{OCINNEIDE}
C.~O'Cinneide and P.~Purdue.
\newblock The {M}/{M}/$\infty$ queue in a random environment.
\newblock {\em J. Appl. Probab}, 23:175--184, 1986.

\bibitem{ROB}
R.~Rebolledo.
\newblock Central limit theorems for local martingales.
\newblock {\em Z. Wahrscheinlichkeitstheorie verw. Gebiete}, 51:269--286, 1980.

\bibitem{ROBE}
Ph. Robert.
\newblock {\em Stochastic Networks and Queues}.
\newblock Springer, Berlin, 2003.

\bibitem{BRUG}
A.~Schwabe, K.~Rybakova, and F.~Bruggeman.
\newblock Transcription stochasticity of complex gene regulation models.
\newblock {\em Biophysical Journal}, 103:1152--1161, 2012.

\bibitem{WHITT}
W.~Whitt.
\newblock Proofs of the martingale {FCLT}.
\newblock {\em Probab. Surveys}, 4:268--302, 2007.

\end{thebibliography}

\end{document}